\documentclass[a4paper,11pt]{article}
\usepackage{amssymb,amsthm,amsmath,latexsym}

\newtheorem{thm}{Theorem}[section]
\newtheorem{pro}[thm]{Proposition}
\newtheorem{lem}[thm]{Lemma}
\newtheorem{cor}[thm]{Corollary}

\newtheorem*{thmA}{Theorem A}
\newtheorem*{thmB}{Theorem B}
\newtheorem*{thmC}{Theorem C}
\newtheorem*{conj}{Conjecture}

\date{}

\title{\normalsize\bf ON GROUPS ADMITTING A WORD\\ WHOSE VALUES ARE ENGEL}

\author{\small{\textsc{Raimundo Bastos} and \textsc{Pavel Shumyatsky}\footnote{This research was supported by CNPq-Brazil.}}\\
\small{Department of Mathematics, University of Brasilia}\\
\small{Brasilia-DF, 70910-900 Brazil}\\
\small{E-mail: bastos@mat.unb.br, pavel@unb.br}\\
[10pt]
\small{\textsc{Antonio Tortora}\footnote{This paper was written while these authors were visiting
the Department of Mathematics of University of Brasilia. They wish to thank the department
for its hospitality and G.N.S.A.G.A. (INdAM) for financial support.} \;and \textsc{Maria Tota}\footnotemark[2]}\\
\small{Dipartimento di Matematica, Universit\`a di Salerno}\\
\small{Via Ponte don Melillo, 84084 - Fisciano (SA), Italy}\\
\small{E-mail: antortora@unisa.it, mtota@unisa.it}}

\begin{document}
\maketitle

\begin{abstract} Let $m,n$ be positive integers, $v$ a multilinear commutator word and $w=v^m$. We prove that if $G$ is a residually finite group in which all $w$-values are $n$-Engel, then the verbal subgroup $w(G)$ is locally nilpotent. We also examine the question whether this is true in the case where $G$ is locally graded rather than residually finite. We answer the question affirmatively in the case where $m=1$. Moreover, we show that if $u$ is a non-commutator word and $G$ is a locally graded group in which all $u$-values are $n$-Engel, then the verbal subgroup $u(G)$ is locally nilpotent.\\

\end{abstract}

\section{Introduction}

Let $n$ be a positive integer and let $x,y$ be elements of a group $G$. The commutators $[x,_n y]$ are defined inductively by the rule
$$[x,_0 y]=x;\quad [x,_n y]=[[x,_{n-1} y],y].$$
An element $x$ is called a (left) Engel element if, for any $g\in G$, there exists $n=n(x,g)\geq 1$ such that $[g,_n x]=1$. If $n$ can be chosen independently of $g$, then $x$ is a (left) $n$-Engel element. A group $G$ is called $n$-Engel if all elements of $G$ are $n$-Engel. It is a long-standing problem whether any $n$-Engel group is locally nilpotent. Following Zelmanov's solution of the restricted Burnside problem \cite{ze1,ze2}, Wilson proved that this is true if $G$ is residually finite \cite{w}. Later the second author showed that if in a residually finite group $G$ all commutators $[x_1,\ldots,x_k]$ are $n$-Engel, then the subgroup $\langle[x_1,\ldots,x_k]\,|\,x_i\in G\rangle$ is locally nilpotent \cite{shu1,shu2}. This suggests the following conjecture.

\begin{conj}\label{verb}
Let $w$ be a group-word and $n$ a positive integer. Assume that $G$ is a residually finite group in which all $w$-values are $n$-Engel. Then the corresponding verbal subgroup $w(G)$ is locally nilpotent.
\end{conj}

Recall that if $w$ is a group-word and $G$ is a group, then the verbal subgroup $w(G)$ of $G$ corresponding to the word $w$ is the subgroup generated by all $w$-values in $G$. Most of the words considered in this paper are {\it multilinear commutators}, also known under the name of {\it outer commutator words}. These are words that have a form of a multilinear Lie monomial, i.e., they are constructed by nesting commutators but using always different variables. For example the word
$$[[x_1,x_2],[y_1,y_2,y_5],z]$$
is a multilinear commutator  while the Engel word
$$[x,y,y,y]$$
is not.

An important family of multilinear commutators consists of the lower central words $\gamma_k$,
given by
\[
\gamma_1=x_1,
\qquad
\gamma_k=[\gamma_{k-1},x_k]=[x_1,\ldots,x_k],
\quad
\text{for $k\ge 2$.}
\]
The corresponding verbal subgroups $\gamma_k(G)$ are the terms of the lower central series of $G$.  Another distinguished sequence of outer commutator words are the derived words $\delta_k$, on $2^k$ variables, which are defined recursively by
\[
\delta_0=x_1,
\quad
\delta_k=[\delta_{k-1}(x_1,\ldots,x_{2^{k-1}}),\delta_{k-1}(x_{2^{k-1}+1},\ldots,x_{2^k})].
\]
The verbal subgroup that corresponds to the word $\delta_k$ is the familiar $k$th derived subgroup of $G$ usually denoted by $G^{(k)}$.

In the present paper we will prove the following theorem.

\begin{thmA}
Let $m,n$ be positive integers, $v$ a multilinear commutator word and $w=v^m$. If $G$ is a residually finite group in which all $w$-values are $n$-Engel, then the verbal subgroup $w(G)$ is locally nilpotent.
\end{thmA}

Similarly to the aforementioned results \cite{w,shu1,shu2} the proof of the above theorem is based on the techniques that Zelmanov created in his solution of the restricted Burnside problem. In \cite{kr}, Kim and Rhemtulla extended Wilson's theorem by showing that any locally graded $n$-Engel group is locally nilpotent. Recall that a group is locally graded if every non-trivial finitely generated subgroup has a proper subgroup of finite index. The class of locally graded groups is fairly large and in particular it contains all residually finite groups. We examine the question whether Theorem A can be extended to the case where $G$ is locally graded rather than residually finite. We answer the question affirmatively in the case where $w$ is a multilinear commutator.

\begin{thmB}
Let $n$ be a positive integer and $w$ a multilinear commutator word. If $G$ is a locally graded group in which all $w$-values are $n$-Engel, then the verbal subgroup $w(G)$ is locally nilpotent.
\end{thmB}

Recall that a non-commutator word $u$ is a word such that the sum of the exponents of some variable involved in it is non-zero. Our next result is about locally graded groups in which all values of a non-commutator word are $n$-Engel.

\begin{thmC}
Let $n$ be a positive integer and $u$ a non-commutator word. If $G$ is a locally graded group in which all $u$-values are $n$-Engel, then the verbal subgroup $u(G)$ is locally nilpotent.
\end{thmC}

In the next section we collect results of general nature that are later used in the proofs of our main theorems. In particular we show that if $G$ is a group having an ascending normal series with locally soluble quotients, then the set of all Engel elements of $G$ coincides with the Hirsch-Plotkin radical of $G$. In Section 3 we describe some important ingredients of what is often called ``Lie methods in group theory''. These are crucial in the proof of Theorem A. Section 4 contains the proofs of the main results.

\section{Preliminary results}

Given subgroups $X$ and $Y$ of a group $G$, we denote by $X^Y$ the smallest subgroup of $G$ containing $X$ and normalized by $Y$.

\begin{lem}\label{H^y}
Let $x$ and $y$ be elements of a group $G$ satisfying $[x,_n y^m]=1$, for some $n,m\geq 1$. Then $\langle x\rangle^{\langle y\rangle}$ is finitely generated.
\end{lem}

\begin{proof}
Set $X=\langle x\rangle^{\langle y^m\rangle}$. Then $X$ is finitely generated by \cite[Exercise 12.3.6]{Rob}. Since ${\langle x\rangle}^{\langle y\rangle}=\langle X^{y^i}\ |\ i=0,\dots,m-1\rangle$, the lemma follows.
\end{proof}

\begin{cor}\label{H^y}
Let $y$ be an element of a group $G$ and $H$ a finitely generated subgroup. If $y^m$ is Engel for some $m\geq 1$, then $H^{\langle y\rangle}$ is finitely generated.
\end{cor}

The following lemma is well-known. We supply the proof for the reader's convenience.

\begin{lem}\label{comm}
If $G$ is a group generated by two elements $x$ and $y$, then $G'=\langle [x,y]^{x^ry^s}\,|\,r,s\in\mathbb{Z}\rangle$.
\end{lem}

\begin{proof}
Let $N=\langle [x,y]^{x^ry^s}\,|\,r,s\in\mathbb{Z} \rangle$. Of course, $N^y$ and $N^{y^{-1}}$ are both contained in $N$. Moreover,
$$[x,y]^{x^r y^s x}=[x,y]^{x^{r+1}y^s [y^s,x]}=[y^s,x]^{-1}[x,y]^{x^{r+1}y^s}[y^s,x].$$
We have $[y^s,x]=[y,x]^{y^{s-1}}[y,x]^{y^{s-2}}\cdots[y,x]$, for all $s\geq 1$. This implies that $N^x\leq N$. Similarly we get $N^{x^{-1}}\leq N$ and so $N$ is normal in $G$. It follows that $G'=N$, as desired.
\end{proof}

\begin{lem}\label{BM}
Let $m\geq 1$ and $G$ be a group generated by a finite subset $X$. If $x^m$ is an Engel element for all $x\in X$, then $G'$ is finitely generated.
\end{lem}

\begin{proof}
First assume that $X=\{x,y\}$. Then $G'=\langle [x,y]^{x^ry^s}\,|\,r,s\in\mathbb{Z}\rangle \rangle$ by Lemma \ref{comm} and we are done since $(\langle [x,y]\rangle^{\langle x\rangle})^{\langle y\rangle}$ is finitely generated by Corollary \ref{H^y}.
Now, let $X=\{x_1,\dots,x_{d+1}\}$ with $d\geq 2$, and suppose that the result is true for subgroups which can be generated by at most $d$ elements of $X$. Set $G_i=\langle x_1,\dots,x_{i-1},x_{i+1},\dots,x_{d+1}\rangle$ for any $i=1,\dots,d+1$.
The induction hypothesis yields that $G_i'$ is finitely generated and so is $(G_i')^{\langle x_i\rangle}$, by Corollary \ref{H^y}.
Straightforward calculation shows that $K=\langle(G_i')^{\langle x_i\rangle}\,|\,i=1,\dots,d+1\rangle$ is a normal subgroup of $G$ and hence $G'=K$. In particular, $G'$ is finitely generated.
\end{proof}

Now, an easy induction gives us the following corollary.

\begin{cor}\label{G^s}
Let $m\geq 1$ and $X$ be a normal commutator-closed subset of a group $G$. Assume that $G$ is generated by finitely many elements of $X$. If $x^m$ is Engel for all $x\in X$, then each term of the derived series of $G$ is finitely generated.
\end{cor}

\begin{proof}
We know that $G'$ is finitely generated by Lemma \ref{BM}. Suppose that $G^{(k)}$ is finitely generated, with $k\geq 1$. Since $X$ is normal commutator-closed, $G^{(k)}$ is generated by finitely many elements of $X$. Thus, Lemma \ref{BM} applies and yields that $G^{(k+1)}$ is finitely generated.
\end{proof}

In any group $G$ there is a unique maximal normal locally nilpotent
subgroup (called the Hirsch-Plotkin radical) containing all normal locally nilpotent subgroups of $G$ \cite[12.1.3]{Rob}.
According to Gruenberg \cite[12.3.3]{Rob} the Hirsch-Plotkin radical of a soluble group is precisely the set of all Engel elements. A straightforward corollary is that the same holds when the group is locally soluble. Next, we extend this result to the class of groups having an ascending normal series with locally soluble factors.

\begin{lem}\label{N}
Let $G$ be a group generated by a set of Engel elements and $H$ a locally soluble normal subgroup of $G$. Then $[H,G]$ is locally nilpotent.
\end{lem}

\begin{proof}
Let $X$ be a set of Engel elements such that $G=\langle X\rangle$. Let $N$ be the subgroup generated by all subgroups of the form $[H,x]$, where $x$ ranges through $X$. Thus $N\leq H$ and since $[H,x]\leq N$, it follows that every $x\in X$ normalizes $N$. Since $G=\langle X\rangle$, we conclude that $N$ is normal in $G$ and hence $N=[H,G]$. The proof of the lemma will be complete once it is shown that $N$ is locally nilpotent. We remark that each of the subgroups $[H,x]$ is normal in $H$. So it is sufficient to show that $[H,x]$ is locally nilpotent for any $x\in X$. Let us show first that $\langle H,x\rangle$ is locally soluble. For any $x\in X$, let $K=\langle h_1,\dots,h_r,x\,|$ $h_j\in H, j=1,\dots,r \rangle$. By Corollary \ref{H^y}, the subgroup $J=\langle h_j\,|\,j=1,\dots,r\rangle^{\langle x\rangle}$ of $H$ is finitely generated and thus soluble. As $J$ is normal in $K$, the subgroup $K$ is soluble. In particular, $\langle H,x\rangle$ is locally soluble. Therefore, by Gruenberg's result, $x$ belongs to the Hirsch-Plotkin radical of $\langle H,x\rangle$. It follows that $[H,x]$ is locally nilpotent, as required.
\end{proof}

\begin{pro}\label{asc}
Let $G$ be a group with an ascending normal series whose factors are locally soluble. Then the set of all Engel elements of $G$ coincides with the Hirsch-Plotkin radical of $G$.
\end{pro}

\begin{proof}
It is enough to prove that the subgroup $E$ generated by all Engel elements of $G$ is locally nilpotent. Clearly, $E$ has an ascending normal series with locally soluble factors. Then Lemma \ref{N} can be applied to each factor of the series giving a refined ascending series whose factors are locally nilpotent. Thus, $E$ is a group having an ascending series with locally nilpotent factors and the claim follows from \cite[Exercise 12.3.7]{Rob}.
\end{proof}

\section{On Lie Algebras Associated with Groups}

Let $L$ be a Lie algebra over a field. We use the left normed notation; thus if $l_1,\dots,l_n$ are elements of $L$ then
$$[l_1,\dots,l_n]=[\dots[[l_1,l_2],l_3],\dots,l_n].$$
An element $a\in L$ is called ad-nilpotent if there exists a positive integer $n$ such that
$$[x,\underbrace{a,\dots,a}_{n \ times}]=0$$
for all $x\in L$. If $n$ is the least integer with the above property then we say that $a$ is ad-nilpotent of index $n$. Let $X$ be any subset of $L$. By a commutator in elements of $X$ we mean any element of $L$  that could be obtained from elements of $X$ by repeated operation of commutation with an arbitrary system of brackets including the elements of $X$. Denote by $F$ the free Lie algebra over the same field as $L$ on countably many free generators $x_1,x_2,\dots$. Let $f=f(x_1,\dots,x_n)$ be a non-zero element of $F$. The algebra $L$ is said to satisfy the identity $f\equiv 0$ if $f(a_1,\dots,a_n)=0$ for any $a_1,\dots,a_n\in L$. In this case we say that $L$ is $PI$. We are now in position to quote a theorem of Zelmanov \cite{ze} which has numerous important applications to group theory.

\begin{thm}\label{Zelmanov}
Let $L$ be a Lie algebra generated by finitely many elements $a_1,a_2,\dots,a_m$ such that all commutators in $a_1,a_2,\dots,a_m$ are ad-nilpotent. If $L$ is $PI$, then it is nilpotent.
\end{thm}

Let $G$ be a group. Given a prime $p$, a Lie algebra can be associated with the group $G$ as follows. We denote by $D_i=D_i(G)$ the $i$th dimension subgroup of $G$ in characteristic $p$ (see for example \cite[Chap. 8]{hb}). These subgroups form a central series of $G$ known as the Zassenhaus-Jennings-Lazard series. Set $L(G)=\oplus D_i/D_{i+1}$. Then $L(G)$ can naturally be viewed as a Lie algebra over the field $\mathbb F_p$ with $p$ elements. For an element $x\in D_i - D_{i+1}$ we denote by $\tilde x$ the element $xD_{i+1}\in L(G)$.

\begin{lem}[Lazard, \cite{l1}]\label{Laz}
For any $x\in G$ we have $(ad\,\tilde{x})^p=ad\,(\tilde{x^p})$.
\end{lem}

An important criterion for a Lie algebra to be $PI$ is the following.

\begin{lem}[Wilson, Zelmanov, \cite{wize}]\label{WZ}
Let $G$ be any group satisfying a group law. Then $L(G)$ is $PI$.
\end{lem}

Let $L_p(G)$ be the subalgebra of $L(G)$ generated by $D_1/D_2$. Important information about the group $G$ can very often be deduced from nilpotency of the Lie algebra $L_p(G)$. In particular, we will require the following result, due to Lazard \cite{l}.

\begin{thm}\label{lazt}
If $G$ is a finitely generated pro-$p$ group such that $L_p(G)$ is nilpotent, then $G$ is $p$-adic analytic.
\end{thm}

\section{Proofs of the main results}

The next lemma is taken from \cite{shu3}.

\begin{lem}\label{delta}
Let $G$ be a group and $v$ a multilinear commutator of weight $k\geq 1$. Then every $\delta_k$-value in $G$ is a $v$-value.
\end{lem}

The following result is a straightforward corollary of Lemma 2.1 in \cite{w}.

\begin{lem}\label{Wilson}
Let $G$ be a finitely generated residually finite-nilpotent group. For each prime $p$, let $R_p$ be the intersection of all normal subgroups of $G$ of finite $p$-power index. If $G/R_p$ is nilpotent for each $p$, then $G$ is nilpotent.
\end{lem}

Now we will deal with Theorem A: {\it Let $m,n$ be positive integers, $v$ a multilinear commutator word and $w=v^m$. If $G$ is a residually finite group in which all $w$-values are $n$-Engel, then the verbal subgroup $w(G)$ is locally nilpotent.}

\begin{pro}\label{residual}
Let $G$ satisfy the hypothesis of Theorem A and assume additionally that $G$ is generated by finitely many Engel elements. Then $G$ is nilpotent.
\end{pro}

\begin{proof}
Since finite groups generated by Engel elements are nilpotent \cite[12.3.7]{Rob}, it follows that $G$ is residually nilpotent. By Lemma \ref{Wilson}, we can assume that $G$ is residually-$p$ for some prime $p$ (i.e., for any non-trivial element $x\in G$ there exists a normal subgroup $N\leq G$ such that $x\notin N$ and $G/N$ is a finite $p$-group). Also, by Lemma \ref{delta}, there exists $k$ such that every $\delta_k$-value is a $v$-value in $G$. Choose arbitrarily finitely many $\delta_k$-values $h_1,\dots,h_d$ and set $H=\langle h_1,\dots,h_d\rangle$. Notice that if $h$ is an arbitrary commutator in $h_1,h_2,\dots,h_d$ with some system of brackets, then $h=\delta_k(x_1,\dots,x_{2^k})$ for suitably chosen $x_1,x_2,\dots,x_{2^k}\in G$.
Let $L=L_p(H)$ be the Lie algebra associated with the Zassenhaus-Jennings-Lazard series
$$H=D_1\geq D_2\geq \cdots$$
of $H$. Then $L$ is generated by $\tilde{h}_i=h_i D_2$, $i=1,2,\dots,d$. Let $\tilde{h}$ be any Lie-commutator in $\tilde{h}_i$ and $h$ be the group-commutator in $h_i$ having the same system of brackets as $\tilde{h}$. By Lemma \ref{Laz}, we have $(ad\,\tilde{h})^m=ad\,(\tilde{h^m})$. Since $h^m$ is $n$-Engel, $\tilde{h}$ is ad-nilpotent of index at most $mn$. Further, $H$ satisfies the identity
$$[y,_n \delta_k^m(x_1,\dots,x_{2^k})]\equiv 1$$
and therefore, by Lemma \ref{WZ}, $L$ satisfies some non-trivial polynomial identity. Now Theorem \ref{Zelmanov} implies that $L$ is nilpotent. Let $\hat{H}$ denote the pro-$p$ completion of $H$. Then $L_p(\hat{H})=L$ is nilpotent and $\hat{H}$ is a $p$-adic analytic by Theorem \ref{lazt}.
Hence, $H$ has a faithful linear representation over the field of $p$-adic numbers. Clearly $H$ cannot have a free subgroup of rank 2 and so, by Tits' Alternative \cite{tits}, $H$ has a soluble subgroup of finite index. It follows that $H$ is soluble (\cite[Exercise 9, p. 129]{DDMS}). Since $h_1,\dots,h_d$ have been chosen arbitrarily, we now conclude that $G^{(k)}$ is locally soluble. Thus, by Proposition \ref{asc}, $G$ is nilpotent.
\end{proof}

\begin{proof}[\bf Proof of Theorem A]
Let $H$ be a finitely generated subgroup of $w(G)$. Clearly, there exist finitely many $w$-values $w_1,\dots,w_d$ such that $H\leq\langle w_1,\dots,$ $w_d\rangle$. Set $W=\langle w_1,\dots,w_d \rangle$. Applying Proposition \ref{residual} to $W$, we obtain that $W$ is nilpotent and this yields that $w(G)$ is locally nilpotent.
\end{proof}

Our attempts to extend Theorem A to the class of locally graded groups so far have been successful only in the case where $w$ is a multilinear commutator, that is, $m=1$.

\begin{thmB}
Let $n$ be a positive integer and $w$ a multilinear commutator word. If $G$ is a locally graded group in which all $w$-values are $n$-Engel, then the verbal subgroup $w(G)$ is locally nilpotent.
\end{thmB}

\begin{proof}
By Lemma \ref{delta} there exists $k$ such that every $\delta_k$-value is a $w$-value. Denote by $X$ the set of all $\delta_k$-values in $G$. Choose finitely many $\delta_k$-values $h_1,\dots,h_d$ and set $H=\langle h_1,\dots,h_d\rangle$. Let $R$ be the intersection of all subgroups of $H$ with finite index and suppose $R\neq 1$. The quotient $H/R$ is residually finite and, by Proposition \ref{residual}, it is nilpotent. It follows that $H^{(s)}\leq R$ for some $s$. Now, $H/H^{(s+1)}$ is nilpotent, since this is a soluble group generated by finitely many Engel elements \cite[12.3.3]{Rob}. In particular, $H/H^{(s+1)}$ is residually finite. Hence $H^{(s)}=H^{(s+1)}$. On the other hand, $X\cap H$ is a normal commutator-closed subset of $H$ and therefore, by Corollary \ref{G^s}, $H^{(s)}$ is finitely generated. Since $R/H^{(s)}$ is a subgroup of the finitely generated nilpotent group $G/H^{(s)}$, we conclude that $R/H^{(s)}$ is finitely generated as well. Thus $R$ is finitely generated, say by $r$ elements. Since $G$ is locally graded, there exists a proper subgroup of $R$ with finite index $t$. By \cite[Theorem 7.2.9]{MHall} the number of subgroups of index $t$ in $R$ is finite and even bounded in terms of $r$ and $t$ only. This implies that the intersection $N$ of all subgroups of $R$ with index $t$ has finite index in $R$. Clearly, $N$ is characteristic and $R/N$ is finite. We deduce that $H/N$ satisfies the maximal condition. But $H/N$ is generated by finitely many Engel elements and so it is nilpotent \cite[12.3.7]{Rob}. Thus $H^{(s+1)}$ is a proper subgroup of $H^{(s)}$, a contradiction. This means that $H$ is residually finite and, by Proposition \ref{residual}, nilpotent. We have shown that every subgroup generated by finitely many $\delta_k$-values is nilpotent and so $G^{(k)}$ is locally nilpotent. By Proposition \ref{asc} $w(G)$ is locally nilpotent.
\end{proof}

Clearly, a quotient of a locally graded group need not be locally graded. We will now quote a useful result, due to Longobardi, Maj, Smith \cite{LMS}, that provides a sufficient condition for a quotient to be locally graded.

\begin{lem}\label{graded}
Let $G$ be a locally graded group and $N$ a normal locally nilpotent subgroup of $G$. Then $G/N$ is locally graded.
\end{lem}

Remind the reader that, by Zelmanov's solution of the restricted Burnside problem, locally graded groups of finite exponent are locally finite (see for example \cite[Theorem 1]{Ma}).

\begin{thmC}
Let $n$ be a positive integer and $u$ a non-commutator word. If $G$ is a locally graded group in which all $u$-values are $n$-Engel, then the verbal subgroup $u(G)$ is locally nilpotent.
\end{thmC}

\begin{proof}
Let $u=u(x_1,\dots, x_r)$ be a non-commutator word. We may assume that the sum of the exponents of $x_1$ is $m\neq 0$. Substitute $1$ for $x_2,\dots,x_r$ and an arbitrary element $g\in G$ for $x_1$. We see that $g^m$ is a $u$-value for every $g\in G$. Thus every $m$th power is $n$-Engel in $G$. Let us show first that $G^m$ is locally nilpotent. Choose arbitrarily finitely many elements $g_1^m,\dots,g_d^m$ and set $H=\langle g_1^m,\dots,g_d^m\rangle$. We need to show that $H$ is nilpotent. Let $R$ be the intersection of all subgroups of $H$ with finite index and suppose $R\neq 1$. Then $H/R$ is residually finite and, by Proposition \ref{residual}, it is nilpotent. In particular, $H^{(s)}\leq R$ for some $s$. Moreover $H^{(s)}$ is finitely generated by Corollary \ref{G^s}. Arguing as in the proof of Theorem B, we get a contradiction. Thus $H$ is residually finite and, by Proposition \ref{residual}, nilpotent. Hence, $G^m$ is locally nilpotent. By Lemma \ref{graded}, it follows that $G/G^m$ is a locally graded group of finite exponent. Therefore $G/G^m$ is locally finite. This yields that $G$ is locally (nilpotent-by-finite). Finally, consider a subgroup $U$ of $G$ generated by finitely many $u$-values. Then $U$ has a nilpotent normal subgroup $N$ of finite index and, by the hypothesis, each $u$-value is Engel. So $U/N$ is nilpotent \cite[12.3.7]{Rob} and, consequently, $U$ is soluble. We deduce that $U$ is nilpotent \cite[12.3.3]{Rob} and $u(G)$ is locally nilpotent.
\end{proof}

\end{document}